\documentclass[10pt,a4paper]{amsart}

\usepackage{amsfonts,amsthm,amssymb,latexsym,amsmath,setspace,amscd,verbatim}
\usepackage[pdftex]{graphicx}
\usepackage{graphics}
\usepackage[matrix,arrow]{xy}
\usepackage[active]{srcltx}
\usepackage[colorlinks=true]{hyperref}
\usepackage{mathrsfs}

\title[Schwartzman-Fried-Sullivan Theory]{A note on Schwartzman-Fried-Sullivan Theory, with an application}
\author[Hryniewicz]{Umberto L. Hryniewicz}
\address{Umberto L. Hryniewicz \\ RWTH Aachen, Jakobstrasse 2, Aachen 52064, Germany}
\email{hryniewicz@mathga.rwth-aachen.de}

\newcommand{\C}{\mathbb{C}}
\newcommand{\R}{\mathbb{R}}
\newcommand{\Z}{\mathbb{Z}}
\newcommand{\N}{\mathbb{N}}
\newcommand{\D}{\mathbb{D}}

\renewcommand{\P}{\mathscr{P}}
\newcommand{\s}{\mathscr{S}}
\newcommand{\Cc}{\mathscr{C}}
\newcommand{\Pp}{\mathscr{P}}

\newcommand{\pr}{{\rm pr}}

\theoremstyle{plain}
\newtheorem{theorem}{\sc Theorem}[section]

\newtheorem{lemma}[theorem]{\sc Lemma}
\newtheorem{corollary}[theorem]{\sc Corollary}

\theoremstyle{definition}
\newtheorem{definition}[theorem]{\sc Definition}

\theoremstyle{remark}
\newtheorem{remark}[theorem]{\sc Remark}

\begin{document}

\maketitle

\begin{abstract}
We prove a theorem on the existence of global surfaces of section with prescribed spanning orbits and homology class. This result is a modification and a refinement of a result due to Fried, recast in terms of invariant measures instead of homology directions.
\end{abstract}

%\tableofcontents

\section{Introduction}

Throughout this paper we fix a smooth flow $\phi^t$ on a smooth closed oriented connected $3$-manifold~$M$. The interplay between the topology of $M$ and the dynamics of $\phi^t$ can be studied by following a program outlined in the 1950s by Schwartzman~\cite{schwartzman}. If $\mu$ is a $\phi^t$-invariant Borel probability measure then $\mu$-almost every point is recurrent and its trajectory almost closes up infinitely often. One may close such ``almost periodic'' long trajectories with short paths, and average in time the value obtained by hitting with a degree one cohomology class $y$. With the help of Ergodic Theory one can show that this procedure defines a $\mu$-integrable function; its integral will be denoted by $\mu\cdot y \in \R$ and called an {\it intersection number} for geometric reasons which will soon become clear. See section~\ref{ssec_defining_intersection_and_rotation_numbers} for precise definitions. This construction can be localized to invariant open subsets. In some sense the values of $\mu\cdot y$, for all possible $\mu$ and~$y$, give a complete portrait of the flow. This claim will be made precise in the context of a particular problem, namely that of deciding when a finite collection of periodic orbits bounds a global surface of section.

The notion of a global surface of section goes back to Poincar\'e's work on the 3-body problem. It is intimately connected to the development of Symplectic Topology: the discovery of Poincar\'e's annulus in the context of Celestial Mechanics led to the statement of his {\it last geometric theorem}~\cite{Po}, nowadays known as the Poincar\'e-Birkhoff theorem, which in turn led Arnold~\cite{arnold_conj} to make his conjectures on the number of fixed points of Hamiltonian diffeomorphisms and of Lagrangian intersections. The Arnold conjectures led to the creation of Floer theory, see for instance~\cite{floer}.

\begin{definition}\label{def_GSS}
A {\it global surface of section} for $\phi^t$ is a compact embedded surface $\Sigma \hookrightarrow M$ such that
\begin{itemize}
\item[(i)] $\partial\Sigma$ is a finite collection of periodic orbits, and $\phi^t$ is transverse to $\Sigma\setminus\partial\Sigma$.
\item[(ii)] For every $p\in M\setminus \partial\Sigma$ there exist $t_+>0$, $t_-<0$ such that $\phi^{t_+}(p)$, $\phi^{t_-}(p)$ belong to $\Sigma$.
\end{itemize}
\end{definition}

\begin{remark}
The case $\partial\Sigma$ is the empty set is not excluded. All our global surfaces of section are oriented by the ambient orientation and the co-orientation induced by the flow: intersection points of trajectories (oriented by the flow) with $\Sigma\setminus\partial\Sigma$ count $+1$.
\end{remark}

Using $\Sigma$ one may study $\phi^t$ in terms of the {\it first return map}, defined by following a point in $\Sigma\setminus\partial\Sigma$ until it hits $\Sigma$ in the future. This point of view opens the door to methods in two-dimensional dynamics.

From now on $L\subset M$ is a fixed null-homologous link consisting of periodic orbits of $\phi^t$, or the empty set. The task at hand is to look for qualitative information that will tell us when $L$ is the boundary of a global surface of section. Besides intersection numbers, other important players are the {\it rotation numbers} of components of $L$ relative to some $y\in H^1(M\setminus L;\R)$. These are roughly described as follows. The linearized flow along a periodic orbit $\gamma \subset L$ induces an orientation preserving diffeomorphism on the circle of oriented rays issuing from a point of~$\gamma$. A class $y$ determines (up to homotopy) an isotopy from the identity to this diffeomorphism. We end up with a real-valued {\it rotation number}, denoted by $\rho^y(\gamma)$. See section~\ref{ssec_defining_intersection_and_rotation_numbers} for a precise discussion.

We denote by $\Pp_\phi(M\setminus L)$ the set of $\phi^t$-invariant Borel probability measures on $M\setminus L$. If $b\in H_2(M,L;\Z)$ then $y^b$ denotes its dual class, seen in $H^1(M\setminus L;\R)$. Recall that, by definition, a Seifert surface is a compact connected orientable embedded surface. We say that $L$ spans the Seifert surface if it is equal to its boundary. Two conventions are useful: the empty set is said to bind an open book decomposition if the ambient manifold fibers over the circle, and a Seifert surface spanned by the empty set is just a closed connected orientable embedded surface.

\begin{theorem}\label{thm_SFS}
Let $b\in H_2(M,L;\Z)$ be induced by an oriented Seifert surface spanned by~$L$. Consider the following assertions.
\begin{itemize}
\item[(i)] $L$ bounds a global surface of section for $\phi^t$ representing the class~$b$.
\item[(ii)] $L$ binds an open book decomposition with connected pages that are global surfaces of section for $\phi^t$ and represent the class~$b$.
\item[(iii)] The following hold:
\begin{itemize}
\item[(a)] $\rho^{y^b}(\gamma)>0$ for every component $\gamma\subset L$.
\item[(b)] $\mu\cdot y^b>0$ for all $\mu\in\Pp_\phi(M\setminus L)$.
\end{itemize}
\end{itemize}
Then (iii) $\Rightarrow$ (ii) $\Rightarrow$ (i). Moreover (i) $\Rightarrow$ (iii) holds $C^\infty$-generically.
\end{theorem}

The case where $L=\emptyset$ and the homology class $b$ is not prescribed is sketched by Ghys in~\cite{ghys}. In fact, the reader will immediately realize that these notes are much inspired by, and owe a lot to, the exposition in~\cite{ghys}.

The case $L\neq\emptyset$ of Theorem~\ref{thm_SFS} is close to work of Fried. The definition of a global surface of section in~\cite[section~5]{fried} is more restrictive than Definition~\ref{def_GSS}. The difference is that in~\cite{fried} one asks for the first return time function to be bounded and the angle between the surface and the vector field not to vanish to first order as one approaches the boundary; let us call these {\it strong} global surfaces of section. It is not hard to construct by hand global surfaces section which are not strong, and the reader is invited to check that our arguments will produce strong ones from assumption (iii). In~\cite[section~5]{fried} one also finds necessary and sufficient conditions for the existence of strong global surfaces of section, but these are stated in terms of assumptions on homology directions in a certain ``blown up'' manifold; such conditions are hard to work with. Our conditions are stated in terms of rotation numbers on the ambient manifold, which are more directly related to the flow and comfortable to work with. The presentation of this new set of sufficient conditions is our first contribution. Another novelty is the identification of a rather simple $C^\infty$-generic assumption that guarantees that the sufficient conditions are also necessary, it reads: {\it For every periodic orbit $\gamma\subset L$, if the rotation number of $\gamma$ with respect to $y^b$ vanishes then $\gamma$ is hyperbolic.} The proof can be found in subsection~\ref{ssec_necessary}.

The strength of Theorem~\ref{thm_SFS} is its generality, it deals with any flow on any oriented $3$-manifold. Its weakness comes from the fact that it might not be easy to check (iii) in concrete examples. However there are restrictive but still very interesting classes of flows for which more applicable existence results can be proved. The basic example is

\begin{theorem}[Birkhoff~\cite{birkhoff}]\label{thm_birkhoff}
Let $c$ be an embedded unit speed closed geodesic on a positively curved Riemannian two-sphere. Then $\dot c \cup -\dot c$ bounds an annulus-like global surface of section for the geodesic flow on the unit tangent bundle.
\end{theorem}

The annulus in Birkhoff's theorem, sometimes called a {\it Birkhoff annulus}, is easy to describe: it consists of the unit vectors based at $c$ pointing towards one of the closed hemispheres determined by $c$. In the light of Theorem~\ref{thm_SFS} all invariant measures in the complement of $\dot c \cup -\dot c$ must intersect positively the Birkhoff annulus, and somehow this is taken care by positivity of the curvature.

There is an important point to be made here: one should work to make the general results from Schwartzman-Fried-Sullivan theory more applicable, and to make restrictive geometric results such as Birkhoff's theorem more general. Pseudo-holomorphic curve theory in symplectizations and symplectic cobordisms, as introduced by Hofer~\cite{93}, implements this program within the class of Reeb flows. The following remarkable theorem is our guide.

\begin{theorem}[Hofer, Wysocki and Zehnder~\cite{convex}]
Every smooth compact strictly convex energy level in a four-dimensional symplectic vector space admits a disk-like global surface of section.
\end{theorem}

We refer to~\cite{icm} for an overview of results obtained in this direction. In~\cite{HSW} new results which can be used to prove Theorem~\ref{thm_birkhoff} quite directly will appear. The reader can find in the book~\cite{FvK_book} by Frauenfelder and van Koert a nice introduction to global surfaces of section in the context of Symplectic Dynamics, with emphasis in Celestial Mechanics. \\

\noindent {\it Acknowledgements.} I thank Alberto Abbondandolo for explaining the proof of a refinement of the Hahn-Banach theorem stated in the appendix. It plays an important role here since in Theorem~\ref{thm_SFS} one looks for global surfaces of section representing a prescribed homology class. I also thank the referee for the valuable feedback and for pointing out corrections. I acknowledge the support by the Friends of the Institute for Advanced Study for funding my research during the academic year 2018-2019 through a von Neumann Fellowship.

%\newpage

\section{Intersection numbers and rotation numbers}\label{ssec_defining_intersection_and_rotation_numbers}

\subsection{Intersection numbers}\label{ssec_int_numbers}

Let $\mathscr{R}_\phi$ denote the set of recurrent points of $\phi^t$.

\begin{lemma}\label{lemma_def_int_numbers}
For every $(\mu,y) \in \Pp_\phi(M\setminus L) \times H^1(M\setminus L;\R)$ there exists a function $f_{\mu,y} \in L^1(\mu)$ with the following property. For $\mu$-almost every point $p\in \mathscr{R}_\phi \setminus L$, if $V$ is a contractible open neighborhood of $p$ in $M\setminus L$, and $T_n \to +\infty$ is a sequence satisfying $\phi^{T_n}(p) \to p$, then
\begin{equation}
f_{\mu,y}(p) = \lim_{n\to\infty} \frac{1}{T_n} \left< y,k(T_n,p) \right>.
\end{equation}
Here $k(T_n,p)$ denotes any loop obtained by concatenating to $\phi^{[0,T_n]}(p)$ a path from $\phi^{T_n}(p)$ to $p$ inside~$V$.
\end{lemma}

\begin{definition}
The above lemma allows one to use the integral 
\begin{equation}
\mu \cdot y = \int_{M\setminus L} f_{\mu,y} \  d\mu
\end{equation}
as a definition of the intersection number\footnote{But see Corollary~\ref{coro_int} for an alternative, perhaps more concrete definition.}.
\end{definition}

We now work towards the proof of Lemma~\ref{lemma_def_int_numbers}. Let $\Omega^1_L \subset \Omega^1(M\setminus L)$ be the subspace defined as follows. Choose a connected component $\gamma \subset L$ and an orientation preserving diffeomorphism
\begin{equation}\label{tubular_map}
\Psi: N \to \R/T\Z \times \D
\end{equation}
defined on some neighborhood $N$ of $\gamma$, satisfying $\Psi(\phi^t(p_0)) = (t,0)$ for some $p_0\in\gamma$. Here $\D\subset\C$ is the closed unit disk oriented by the complex orientation of $(\C,i)$, $T>0$ is the primitive period of $\gamma$, $\R/T\Z$ is oriented by the canonical orientation of the real line, and $\R/T\Z \times \D$ is oriented as a product. On $N\setminus\gamma$ we have coordinates
\begin{equation}\label{polar_tubular_coordinates}
(t,r,\theta) \in \R/T\Z \times (0,1]\times\R/2\pi\Z
\end{equation}
via the identification $\Psi^{-1}(t,re^{i\theta}) \simeq (t,r,\theta)$, which will be referred to as {\it tubular polar coordinates for $\gamma$}. Define $\Omega^1_L$ to be the set of smooth $1$-forms on $M\setminus L$ that can be represented as $Adt + Bdr + Cd\theta$ with bounded coefficients $A,B,C$ with respect to tubular polar coordinates $(t,r,\theta)$ around each connected component of~$L$.

\begin{remark}
The space $\Omega^1_L$ does not depend on choices of tubular polar coordinates.
\end{remark}

The proof of Lemma~\ref{lemma_def_int_numbers} uses the following lemmas.

\begin{lemma}\label{lemma_bounded_intersection_SFS}
If $\beta\in\Omega^1_L$ then $\beta(X)$ is bounded.
\end{lemma}

\begin{lemma}\label{lemma_nice_representative}
Any $y\in H^1(M\setminus L;\R)$ can be represented by a closed $1$-form in $\Omega^1_L$.
\end{lemma}

\begin{proof}[Proof of Lemma~\ref{lemma_def_int_numbers}]

Let $\bar S_{\mu,y}$ be the set of points $p \in \mathscr{R}_\phi \setminus L$ with the following property: There exists some open contractible neighborhood $V\subset M\setminus L$ of $p$ such that for every sequence $T_n \to +\infty$ satisfying $\phi^{T_n}(p) \to p$ the limit
\begin{equation}\label{limit_natural}
\lim_{n\to+\infty} \frac{1}{T_n} \left< y,k(T_n,p) \right>
\end{equation}
exists. Here $k(T_n,p)$ denotes, for $n$ large enough, any loop obtained by concatenating to $\phi^{[0,T_n]}(p)$ a path from $\phi^{T_n}(p)$ to $p$ inside $V$. By contractibility of $V$, the above limits do not depend on choice of closing paths, but they could in principle depend on the sequence $T_n$. Clearly the existence and the values of the above limits do not depend on $V$.

By Lemma~\ref{lemma_nice_representative} we can choose a closed representative $\beta \in \Omega^1_L$ of $y$. Then $\beta(X)$ is bounded, by Lemma~\ref{lemma_bounded_intersection_SFS}. By the ergodic theorem there exists a Borel set $S\subset M\setminus L$ with the following properties: $\mu(S)=1$, and for every $p\in S$ the limit
\[
\lim_{T\to\infty} \frac{1}{T} \int_0^T \beta(X) \circ \phi^t(p) \ dt = \lim_{T\to\infty} \frac{1}{T} \int_{\phi^{[0,T]}(p)} \beta 
\]
exists, and defines a $\mu$-integrable function satisfying
\[
\int_{S} \left( \lim_{T\to\infty} \frac{1}{T} \int_{\phi^{[0,T]}(p)} \beta \right) \ d\mu = \int_{M\setminus L} \beta(X) \ d\mu.
\]

Let $p\in S \cap \mathscr{R}_\phi$, choose any contractible open neighborhood $V\subset M\setminus L$ of~$p$, and let $T_n \to +\infty$ be any sequence satisfying $\phi^{T_n}(p) \to p$. Fix any auxiliary Riemannian metric $g$ and, for $n$ large enough, consider the loop $k'(T_n,p)$ obtained by concatenating to $\phi^{[0,T_n]}(p)$ a path from $\phi^{T_n}(p)$ to $p$ inside $V$ of $g$-length not larger than $1$. Since the $g$-norm of $\beta$ is bounded near $p$, we get
\begin{equation}
\begin{aligned}
\lim_{n\to\infty} \frac{1}{T_n} \int_{\phi^{[0,T_n]}(p)} \beta &= \lim_{n\to\infty} \frac{1}{T_n} \left( \left< y,k'(T_n,p) \right> + O(1) \right) \\
& = \lim_{n\to\infty} \frac{1}{T_n} \left< y,k'(T_n,p) \right>
\end{aligned}
\end{equation}
Since $V$ is contractible, the limit obtained by replacing the loop $k'(T_n,p)$ on the right-hand side by any $k(T_n,p)$ exists, and is again equal to $\lim_{T\to\infty} \frac{1}{T} \int_{\phi^{[0,T]}(p)} \beta$.

The above argument proves two facts. The first is that $S \cap \mathscr{R}_\phi \subset \bar S_{\mu,y}$. The second is that for every $p\in S \cap \mathscr{R}_\phi$ the values of limits as in~\eqref{limit_natural} do not depend on a particular sequence~$T_n$, and they define a function $f_{\mu,y} \in L^1(\mu)$ satisfying $$ \int_{M\setminus L} f_{\mu,y} \ d\mu = \int_{M\setminus L} \beta(X) \ d\mu. $$
\end{proof}

\begin{corollary}[of the proof]\label{coro_int}
If $\beta \in \Omega^1_L$ represents $y\in H^1(M\setminus L;\R)$ and $\mu \in \Pp_{\phi}(M\setminus L)$ then $\mu \cdot y = \int_{M\setminus L} \beta(X) \ d\mu$.
\end{corollary}

We end with proofs of lemmas~\ref{lemma_bounded_intersection_SFS} and~\ref{lemma_nice_representative}.

\begin{proof}[Proof of Lemma~\ref{lemma_bounded_intersection_SFS}]
Let $(t,x+iy) \in \R/T\Z\times \D$ be coordinates on a neighborhood of a connected component $\gamma \subset L$ given by a map as in~\eqref{tubular_map}. Let $re^{i\theta}=x+iy$ be polar coordinates on $\D$. Writing $X = X_1\partial_t + X_2\partial_x+X_3\partial_y$, observe that $X_2(t,0)=X_3(t,0)=0$, hence $|X_2|=O(r)$, $|X_3|=O(r)$ as $r\to0$. We get
\begin{equation*}
d\theta(X) = \frac{xX_3-yX_2}{r^2} = O(1) \qquad dr(X) = \frac{xX_2+yX_3}{r} = O(r)
\end{equation*}
as $r\to0$. The conclusion follows since, by assumption, $\beta = Adt + Bdr + Cd\theta$ where $A$, $B$ and $C$ are $O(1)$ as $r\to0$.
\end{proof}

\begin{proof}[Proof of Lemma~\ref{lemma_nice_representative}]
Choose any closed $1$-form $\beta_0$ in $M\setminus L$ representing $y$. In tubular polar coordinates $(t,r,\theta)$ around a connected component $\gamma$ of $L$ we can write $\beta_0$ in the frame $\{dt,dr,d\theta\}$ as $pdt + qd\theta + df$ where $p,q$ are constants depending only on $y$. Consider a smooth function $\rho:(0,1)\to[0,1]$ satisfying $\rho(s)=0$ if $s\sim0$, $\rho(s)=1$ if $s\sim1$, $\rho'$ has compact support. Consider the smooth closed $1$-form agreeing with $\beta_0$ far from $\gamma$, and with $pdt + qd\theta + d(\rho f)$ near $\gamma$. It still represents $y$ since it differs from $\beta_0$ by $d((1-\rho)f)$. Repeating this process near each connected component of $L$ we obtain the desired representative in $\Omega^1_L$.
\end{proof}

\subsection{Rotation numbers}

Fix a connected component $\gamma \subset L$, and consider the trivial vector bundle $E_\gamma = TM|_\gamma/T\gamma$ over $\gamma$. Let $T>0$ be the primitive period of $\gamma$. Coordinates $(t,x+iy=re^{i\theta})$, with $t \in \R/T\Z$ and $x+iy \in\D$, defined on a neighborhood of $\gamma$ by a diffeomorphism as in~\eqref{tubular_map} induce a trivializing frame $\{\partial_x,\partial_y\}$ on~$E_\gamma$, and a vector bundle isomorphism $E_\gamma \simeq \R/T\Z \times \R^2$. This frame induces an angular fiber coordinate on the circle bundle $(E_\gamma \setminus 0)/\R_+$ still denoted by $\theta \in \R/2\pi\Z$ with no fear of ambiguity. We end up with a bundle isomorphism $(E_\gamma\setminus 0)/\R_+ \simeq \R/T\Z \times \R/2\pi\Z$, with coordinates $(t,\theta)$. The linearized flow $d\phi^t$ on $(E_\gamma\setminus 0)/\R_+$ gets represented as the flow of a vector field of the form
\begin{equation}\label{form_of_ODE_linearized_flow}
\partial_t+b(t,\theta)\partial_\theta
\end{equation}
on this torus. Seeing $b(t,\theta)$ as a $T\Z\times 2\pi\Z$ periodic function on $\R^2$, this flow lifts to a flow on $\R^2$ of the form $t \cdot (t_0,\theta_0) \mapsto (t_0+t,\theta(t+t_0;t_0,\theta_0))$, where $\theta(t;t_0,\theta_0) \in \R$ solves the initial value problem
\begin{equation*}
\dot\theta(t;t_0,\theta_0) = b(t,\theta(t;t_0,\theta_0)) \qquad \theta(t_0;t_0,\theta_0) = \theta_0.
\end{equation*}

\begin{definition}
If $y$ is cohomologous to $pdt+qd\theta$ near $\gamma$ then we define
\begin{equation}\label{def_formula_rot_number}
\rho^y(\gamma) = \frac{T}{2\pi} \left( p+q\lim_{t\to+\infty}\frac{\theta(t;t_0,\theta_0)}{t} \right).
\end{equation}
\end{definition}

\begin{remark}
The argument in~\cite[pp. 104-105]{arnold} proves that the limit in~\eqref{def_formula_rot_number} exists, does not depend on $(t_0,\theta_0)$, and the convergence is uniform in $(t_0,\theta_0)$. 
\end{remark}

The task now is to show that $\rho^y(\gamma)$ is independent of choice of coordinates. Another such choice induces new coordinates $(t',\theta') \in \R/T\Z \times \R/2\pi\Z$ on the circle bundle $(E_\gamma\setminus 0)/\R_+$ and, moreover, $t'=t$. The degree of $$ t\in\R/T\Z \to \theta'(t,\theta) \in \R/2\pi\Z $$ is independent of $\theta$, and denoted by $m\in\Z$. Hence the class $y$ is cohomologous to $(p-2\pi mq/T)dt+qd\theta'$ near $\gamma$ and 
\[
\lim_{t\to+\infty} \frac{\theta'(t;t_0,\theta_0')}{t} = \frac{2\pi}{T} m + \lim_{t\to+\infty} \frac{\theta(t;t_0,\theta_0)}{t}.
\]
With these new choices we would then have defined $\rho^y(\gamma)$ to be 
\begin{equation*}
\begin{aligned}
& \frac{T}{2\pi} \left( p-\frac{2\pi}{T}mq + q \lim_{t\to+\infty} \frac{\theta'(t;t_0,\theta_0')}{t} \right) \\
&= \frac{T}{2\pi} \left( p-\frac{2\pi}{T}mq + q\left( \frac{2\pi}{T} m + \lim_{t\to+\infty}\frac{\theta(t;t_0,\theta_0)}{t} \right) \right) \\
&= \frac{T}{2\pi} \left( p+q\lim_{t\to+\infty}\frac{\theta(t;t_0,\theta_0)}{t} \right)
\end{aligned}
\end{equation*}
as desired.

\section{Proof of Theorem~\ref{thm_SFS}}

Note first that (ii) $\Rightarrow$ (i) follows from definitions.

\subsection{Blowing periodic orbits up}\label{ssec_blowing_up}

Enumerate the components $\gamma_1,\dots,\gamma_h$ of $L$, and denote their primitive periods by $T_j>0$. For each $j\in\{1,\dots,h\}$ choose a small neighborhood $N_j$ of $\gamma_j$ and an orientation preserving diffeomorphism
\begin{equation}\label{tubes_aligned}
\Psi_j:N_j \to \R/T_j\Z\times \D
\end{equation}
as in~\eqref{tubular_map} which is {\it aligned with $b$}, i.e. the loop $t\in \R/T_j\Z \mapsto \Psi_j(t,1)\in M\setminus L$ has zero algebraic intersection number with $b$, and for some (hence any) $t_0\in\R/T_j\Z$ the algebraic intersection number of the loop $\theta \in \R/2\pi\Z \mapsto \Psi_j(t_0,\theta) \in M\setminus L$ belongs to $\{1,-1\}$.

\begin{remark}
Such a choice is only possible since it is assumed in Theorem~\ref{thm_SFS} that $b$ comes from a Seifert surface for $L$. This is an absolutely crucial choice which the reader must keep in mind. It will not be used in this paragraph, but will play an important role in the proof of (iii) $\Rightarrow$ (ii) in Theorem~\ref{thm_SFS}.
\end{remark}

The $\Psi_j$ induce tubular polar coordinates $(t,r,\theta) \in \R/T_j\Z \times (0,1] \times \R/2\pi\Z$ around the $\gamma_j$, as explained in~\ref{ssec_int_numbers}. A smooth $3$-manifold $M_L$ can be constructed by blowing $L$ up, more precisely it is defined as
\begin{equation}\label{blown_up_manifold}
M_L := \left. \left\{ M\setminus L \ \ \sqcup \ \ \bigsqcup_{j=1}^h \R/T_j\Z \times (-\infty,1]\times \R/2\pi\Z \right\} \right/ \sim
\end{equation}
where, for each $j\in\{1,\dots,h\}$, the point $(t,r,\theta) \in \R/T_j\Z \times (0,1] \times \R/2\pi\Z$ is identified with the point $\Psi^{-1}_j(t,re^{i\theta}) \in N_j\setminus \gamma_j \subset M\setminus L$.

From now on we fix $j$, work on $N_j$, and consider $Z=(\Psi_j)_*X = Z(t,x+iy)$. Then $Z$ is a vector field on $\R/T_j\Z \times \D \subset \R/T_j\Z \times \C$, and as such it can be seen as smooth function $Z:\R/T_j\Z \times \D \to \R \times \C$. Write $x+iy=re^{i\theta}$ and consider the smooth map 
\begin{equation}\label{map_Phi}
\begin{aligned}
& \Phi: \R/T_j\Z \times[0,1]\times\R/2\pi\Z \to \R/T_j\Z \times \D \\
& \Phi(t,r,\theta) = (t,re^{i\theta})
\end{aligned}
\end{equation}
Then $\Phi^{-1}$ is well-defined and smooth on $\R/T_j\Z \times \D\setminus0$, and
\begin{equation*}
W = \Phi^*(Z|_{\R/T_j\Z\times (\D\setminus\{0\})})
\end{equation*}
is a smooth vector field on $\R/T_j\Z\times (0,1]\times\R/2\pi\Z$. We claim that there exists a smooth extension $\widehat W$ of $W$ to $\R/T_j\Z \times (-\infty,1]\times\R/2\pi\Z$ such that $\widehat W$ is tangent to the torus $\R/T_j\Z \times\{0\}\times\R/2\pi\Z$. Let us prove this. In the frame $\{\partial_t,\partial_r,\partial_\theta\}$ we have
\[
D\Phi(t,r,\theta)^{-1} = \begin{bmatrix} 1 & 0 & 0 \\ 0 & \cos\theta & \sin\theta \\ 0 & -r^{-1}\sin\theta & r^{-1}\cos\theta \end{bmatrix}
\]
Note also that $Z(t,re^{i\theta}) = \begin{pmatrix} 1 \\ 0 \end{pmatrix} + A(t,re^{i\theta})  re^{i\theta}$ where 
\[
A(t,re^{i\theta}) = \int_0^1 D_2Z(t,\tau re^{i\theta}) \ d\tau
\]
is smooth (the second component of $Z$ is a complex number). Moreover, 
\begin{equation}\label{DZ_at_orbit}
DZ(t,0) = \begin{pmatrix} 0 & A_1(t,0) \\ 0 & A_2(t,0) \end{pmatrix} \qquad \text{with} \qquad A(t,0) = \begin{pmatrix} A_1(t,0) \\ A_2(t,0) \end{pmatrix}.
\end{equation}
Hence, if we see the second component as an element of $\R^2$, we get
\begin{equation}\label{expansion_W}
\begin{aligned}
W(t,r,\theta) &= D\Phi(t,r,\theta)^{-1}  Z(t,re^{i\theta}) \\
&= \begin{pmatrix} 1 \\ 0 \end{pmatrix} + \begin{bmatrix} r & 0  \\ 0 & \begin{pmatrix} r\cos\theta & r\sin\theta \\ -\sin\theta & \cos\theta \end{pmatrix} \end{bmatrix}  A(t,re^{i\theta})  e^{i\theta}
\end{aligned}
\end{equation}
which, as the reader will immediately see, is smooth all the way up to $r=0$. Hence it has a smooth extension $\widehat W$ to $\R/T\Z \times (-\infty,1] \times \R/2\pi\Z$, and at $r=0$ this vector field has no component in $\partial_r$. In fact, we see from the above formula that
\begin{equation}\label{vector_field_on_boundary}
\widehat W(t,0,\theta) = \partial_t + b(t,\theta)\partial_\theta \qquad b(t,\theta) = \left< A_2(t,0)e^{i\theta},ie^{i\theta} \right>
\end{equation}
It follows that $X|_{M\setminus L}$ can be smoothly extended to a vector field 
\begin{equation}\label{extended_vector_field}
X_L \in \mathscr{X}(M_L)
\end{equation}
tangent to the tori
\begin{equation}\label{boundary_tori}
\Sigma_j = \R/T_j\Z\times\{0\}\times\R/2\pi\Z
\end{equation}
Of course, this extension is not unique, but the restriction of $X_L$ to the closure $D_L$ of $M\setminus L$ in $M_L$ is unique. The vector field $X$ does not vanish near $L$. In view of~\eqref{vector_field_on_boundary} the extension $X_L$ can be chosen to generate a complete flow and to have no zeros on $M_L\setminus D_L$. We will also denote the flow of $X_L$ by $\phi^t$ with no fear of ambiguity.

\begin{remark}
The linearized flow $d\phi^t|_{(0,0)} (s_0,u_0)$ along $\gamma_j$ in coordinates $(t,x+iy)$ induced by $\Psi_j$ is the solution of the initial value problem
\begin{equation*}
\begin{pmatrix} \dot s(t) \\ \dot u(t) \end{pmatrix} = DZ(t,0) \begin{pmatrix} s(t) \\ u(t) \end{pmatrix} \qquad \begin{pmatrix} s(0) \\ u(0) \end{pmatrix} = \begin{pmatrix} s_0 \\ u_0 \end{pmatrix}
\end{equation*}
From~\eqref{DZ_at_orbit} we see that $u(t)$ satisfies the linear equation $\dot u(t) = A_2(t,0) u(t)$. With the aid of $\Psi_j$ the coordinates $x+iy$ are precisely induced by the frame which is used to write down the vector field~\eqref{form_of_ODE_linearized_flow}. This shows that the matrix $M(t)$ representing the linearized flow on $E_{\gamma_j}$ satisfies $\dot M(t)=A_2(t,0)M(t)$. Thus, $M(t)u_0 = u(t)$, and in polar coordinates $u(t) = r(t)e^{i\theta(t)}$ one easily computes
\begin{equation}
\dot \theta(t) = \left< A_2(t,0)e^{i\theta(t)},ie^{i\theta(t)} \right> = b(t,\theta(t)).
\end{equation}
This gives a concrete formula for the function $b(t,\theta)$ appearing in~\eqref{form_of_ODE_linearized_flow}.
\end{remark}

\subsection{Schwartzman cycles and structure currents}

The manifold $M_L$~\eqref{blown_up_manifold} was obtained by blowing $L$ up, and the smooth domain $D_L \subset M_L$ was defined to be the closure of $M\setminus L$ in $M_L$. We see that 
\begin{equation*}
D_L = \left. \left\{ M\setminus L \ \ \sqcup \ \ \bigsqcup_{j=1}^h \R/T_j\Z \times [0,1]\times \R/2\pi\Z \right\} \right/ \sim
\end{equation*}
has boundary equal to $\partial D_L = \bigsqcup_{j=1}^h \Sigma_j$ where the $\Sigma_j$ are the tori~\eqref{boundary_tori}. The smooth vector field $X_L$~\eqref{extended_vector_field} restricts to $D_L$ as the unique continuous extension of $X$ from $M\setminus L$ to $D_L$, and it is tangent to $\partial D_L$.

In~\cite{derham} de Rham equips $\Omega^1(M_L)$ with the $C^\infty_{\rm loc}$-topology and defines a $1$-current with compact support as an element of its topological dual $C_1 = \Omega^1(M_L)'$. The space $C_1$ is equipped with its weak* topology. It is a useful fact, proved in~\cite[\S 17]{derham}, that the map $\Omega^1(M_L) \to C_1'$ (topological dual of $C_1$ with its weak* topology) given by $\omega \mapsto \left<\cdot,\omega\right>$ is a linear homeomorphism; in other words $\Omega^1(M_L)$ is reflexive.

More generally, one may consider the space $C_p$ of $p$-currents with compact support, defined as the topological dual of $\Omega^p(M_L)$ equipped with the $C^\infty_{\rm loc}$-topology. As before $C_p$ is equipped with its weak* topology. The boundary operator $$ \partial:C_{p+1} \to C_p $$ is defined as the adjoint of the exterior derivative $d:\Omega^p(M_L) \to \Omega^{p+1}(M_L)$. A current in $C_p$ is called a cycle if it is in the kernel of $\partial :C_p \to C_{p-1}$, and is called a boundary if it is in the image of $\partial: C_{p+1} \to C_p$. The space of boundaries is denoted by $B_p$, and the space of cycles by $Z_p$. It turns out that $H^p(M_L;\R) = Z_p/B_p$.

Consider the set $\P$ of positive compactly supported finite Borel measures on~$M_L$, and let
\begin{equation}\label{Schwartzman_cycles_set}
\P_{X_L}(D_L) \subset \P
\end{equation}
be the subset of those which are $X_L$-invariant probability measures supported on~$D_L$. Any $\mu\in\P$ defines a $1$-current $c_\mu \in C_1$ by the formula 
\[
\left< c_\mu,\omega \right> = \int_{M_L} \omega(X_L) \ d\mu, \qquad \omega\in\Omega^1(M_L)
\]
We follow Sullivan's notation and write $c_\mu = \int_{M_L} X_L \ d\mu$. A simple calculation shows that if $\mu \in \P$ is supported in $D_L$ then $c_\mu$ is a cycle if, and only if, $\mu$ is $X_L$-invariant. The elements of the set 
\begin{equation}
\s_X := \{ c_\mu \mid \mu\in \P_{X_L}(D_L) \} \subset Z_1
\end{equation}
will be called {\it Schwartzman cycles}.

Consider {\it Dirac currents} $\delta_p\in C_1$, $p\in M_L$, defined by $\left<\delta_p,\omega\right> = \omega(X_L)|_p \in \R$. Let $\Cc\subset C_1$ denote the closed convex cone generated by $\{\delta_p \mid p\in D_L\}$. In~\cite{sullivan} $\Cc$ is called the {\it cone of structure currents in $D_L$}.

\subsubsection{Compactness}\label{sssec_compactness}

\begin{lemma}
The following hold.
\begin{itemize}
\item[(I)] There exists $\omega \in \Omega^1(M_L) = C_1'$ such that $\left< c,\omega \right>>0$ for every $c\in \Cc\setminus\{0\}$.
\item[(II)] If $\omega$ is as in (I) then the convex set $K = \{ c\in\Cc \mid  \left< c,\omega \right>=1\}$ is compact. 
\end{itemize}
\end{lemma}

\begin{proof}[Proof of (I)]
We first claim that for every neighborhood $\mathcal{O}$ of $0$ in $C_1$ there exists $\delta > 0$ such that if $p_1,\dots,p_N \in D_L$ and $a_1,\dots,a_N>0$ satisfy $\sum_i a_i < \delta$ then $\sum_i a_i\delta_{p_i} \in \mathcal{O}$. Here $N\geq1$ is arbitrary. In fact, by the definition of weak* topology we find $\eta_1,\dots,\eta_J \in \Omega^1(M_L)$ and $\epsilon>0$ such that 
\[
\mathcal{V} := \left\{ c\in C_1 \ \text{such that} \ \max_j |\left<c,\eta_j\right>|<\epsilon \right\} \subset \mathcal{O}.
\]
Fix $R > \max_j \|\eta_j(X_L)\|_{L^\infty(D_L)}$ and set $\delta = \epsilon/R$. Consider a finite sum $\sum_i a_i\delta_{p_i}$ where the $p_i$ are points of $D_L$ and the $a_i>0$ satisfy $\sum_ia_i<\delta$. Then
\[
\left| \left< \sum_i a_i\delta_{p_i},\eta_j \right> \right| \leq \sum_i a_i \|\eta_j(X_L)\|_{L^\infty(D_L)} < R \sum_ia_i < R\delta = \epsilon
\]
implying that $\sum_{i=1}^N a_i\delta_{p_i} \in \mathcal{V} \subset \mathcal{O}$ as desired.

Now choose $\omega\in\Omega^1(M_L)$ such that $\omega(X_L)>0$ pointwise on $D_L$, and choose $d>0$ such that $d < \omega(X_L)|_p$ for every $p \in D_L$. We claim that $\left< c,\omega \right>>0$ for every $c\in \Cc \setminus \{0\}$. In fact, fix any $c\in \Cc \setminus \{0\}$ arbitrarily. Since $C_1$ is Hausdorff, we can find $\mathcal{O}'$ neighborhood of $c$ and $\mathcal{O}$ neighborhood of $0$ such that $\mathcal{O}' \cap \mathcal{O} = \emptyset$. By what is proved above we can find $\delta>0$ such that if $\sum_{i=1}^Na_i\delta_{p_i}$ with $a_i>0$ satisfies $\sum_i a_i < \delta$ and $p_i \in D_L$ then $\sum_{i=1}^Na_i\delta_{p_i} \in \mathcal{O}$ (here $N\in\N$ is arbitrary). Hence if $\sum_{i=1}^Na_i\delta_{p_i} \in \mathcal{O}'$ satisfies $p_i \in D_L$ and $a_i>0$ for all $i$, then $\sum_ia_i \geq \delta$ and we can conclude that  
\[
\sum_{i=1}^Na_i\delta_{p_i} \in \mathcal{O}' \Rightarrow \left< \sum_{i=1}^Na_i\delta_{p_i},\omega \right> \geq d \sum_ia_i \geq d \delta.
\]
By the definition of $\Cc$ we get $\left< c,\omega \right> \geq \delta d > 0$, as desired.
\end{proof}

\begin{proof}[Proof of (II)]
Define the convex set $K = \{ c\in\Cc \mid  \left< c,\omega \right>=1\}$ where $\omega$ is a $1$-form satisfying (I). Note that $\omega(X)$ is pointwise strictly positive over $D_L$ since for all $p\in D_L$ we have $\delta_p \in \Cc \setminus\{0\}$. 

The set $K$ is closed since it is the intersection of two closed sets. If we can show that $K$ is contained on a compact subset of $C_1$ then it will follow that $K$ is compact. We claim that the set $\{\left<c,\eta\right> \mid c\in K\}\subset\R$ is bounded, for every $\eta \in \Omega^1(M_L)$. By the definition of $\Cc$ one needs only to consider the case where $c\in K$ is a finite linear combination of Dirac currents at points of $D_L$ with positive coefficients. There exists a constant $A>0$, depending on $\eta$ and $\omega$ such that $|\eta(X_L)| \leq A\omega(X_L)$ holds pointwise on the compact set $D_L$. If $c = \sum_{i=1}^N a_i \delta_{p_i} \in K$, $a_i>0$ and $p_i\in D_L$, then
\[
\left| \left< \sum_{i=1}^N a_i \delta_{p_i},\eta \right> \right| \leq \sum_{i=1}^N a_i|\eta(X_L)_{p_i}| \leq A \sum_{i=1}^N a_i \omega(X_L)_{p_i} = A \left<\sum_{i=1}^N a_i \delta_{p_i},\omega\right> = A
\]
as desired. We apply the Banach-Steinhaus Theorem~\cite[2.5]{rudin} to conclude that $K$ is an equicontinuous set of linear functionals, namely there exists a neighborhood $\mathscr{V}$ of $0$ in $\Omega^1(M_L)$ such that $$ \{ \left<c,\eta\right> \mid c\in K,\ \eta\in\mathscr{V} \} \subset [-1,1]. $$ This means that $K\subset \mathscr{K}_\mathscr{V}$ where 
\begin{equation*}
\mathscr{K}_\mathscr{V} = \{ c \in C_1 \mid \left<c,\eta\right> \in [-1,1] \ \forall \eta \in \mathscr{V} \}.
\end{equation*}
Now the Banach-Alaoglu theorem~\cite[3.15]{rudin} asserts that $\mathscr{K}_\mathscr{V}$ is compact. Compactness of $K$ follows.
\end{proof}

\subsubsection{Representation by measures}\label{sssec_representing_by_measures}

\begin{lemma}
For every $c\in\Cc$ there exists a unique finite Borel measure on $D_L$ such that $c=\int_{D_L}X_L \ d\mu$.
\end{lemma}

\begin{proof}
Fix $c\in\Cc$. Let $q\in D_L$ and choose a coordinate system $x_1,x_2,x_3$ defined on an open relatively compact neighborhood $V$ of $q$ in $M_L$, taking values on the open set $U\subset \R^3$, such that $X_L = \partial_{x_1}$. Any $1$-form $\omega$ compactly supported in $V$ can be written as $\omega = h_1 dx_1 + h_2 dx_2 + h_3 dx_3$, and $\left<c,\omega\right> = \left<c,h_1 dx_1\right>$. This last assertion follows from the density in $\Cc$ of finite linear combinations of Dirac currents.

The map $f\mapsto \left< c,fdx_1 \right>$ defines a distribution on $U$ of order $0$; here~$f$ stands for a test function on $U$. This means that $\left< c,f_ndx_1 \right> \to 0$ holds for any sequence $f_n$ of test functions on $U$ supported on a common compact subset $F\subset U$, and satisfying $f_n\to0$ uniformly. This is true since it is clearly true for finite combinations of Dirac measures, and hence also for currents in their closure. Similarly we conclude that $\left<c,fdx_1\right>\geq 0$ holds for any non-negative test function on $U$.

We have checked that we can apply the Riesz representation theorem to find a unique positive Borel measure $\mu$ on $U$, finite on compact subsets of $U$, such that $\left<c,fdx_1\right> = \int_U f \ d\mu$ for all compactly supported continuous functions $f:U\to\R$. Pushing forward to $V$ we get a unique Borel measure on $V$, supported on $V\cap D_L$ and still denoted by $\mu$, such that $\left<c,\omega\right> = \int_V \omega(X_L) \ d\mu$ holds for all $\omega \in \Omega^1(M_L)$ which is compactly supported on $V$. The uniqueness property allows us to patch such local constructions to obtain a unique Borel measure $\mu$ on $M_L$, supported on $D_L$, satisfying
\begin{equation}
\left<c,\omega\right> = \int_{M_L} \omega(X_L) d\mu
\end{equation}
for all $\omega\in\Omega^1(M_L)$. The total $\mu$-measure of $M_L$ is finite since compact subsets have finite measure and $\mu$ is supported in $D_L$.
\end{proof}

\subsection{Proof of (iii) $\Rightarrow$ (ii)}

We only deal with the case $L\neq \emptyset$, the case $L=\emptyset$ is easier and left to the reader. Recall the chosen tubular neighborhoods $N_j$ of the connected components $\gamma_j\subset L$, used to blow $L$ up and obtain the manifold $M_L$~\eqref{blown_up_manifold} as explained in~\ref{ssec_blowing_up}. These were equipped with tubular polar coordinates $(t,r,\theta)$ in $\R/T_j\Z \times (0,1]\times \R/2\pi\Z$ on $N_j\setminus \gamma_j$ in such a way that the boundary component of $D_L$ corresponding to $\gamma_j$ is the torus $\Sigma_j = \R/T_j\Z \times 0\times \R/2\pi\Z$.

Assume that $b$ satisfies (iii) in Theorem~\ref{thm_SFS}. Since the $\Psi_j$ are aligned with $b$ as explained in~\ref{ssec_blowing_up}, we can choose a closed $1$-form $\beta\in\Omega^1(M_L)$ that represents the class $y^b$ in $M\setminus L$ and is written as
\begin{equation}\label{numbers_epsilon_j}
\beta = \frac{\epsilon_j}{2\pi} \ d\theta
\end{equation}
on the end $\R/T_j\Z \times (-\infty,0]\times \R/2\pi\Z$ corresponding to $\gamma_j$, for some $\epsilon_j \in \{1,-1\}$. In particular, the restriction of $\beta$ to $M\setminus L$ belongs to $\Omega^1_{L}$. For simplicity we may just write $y$ instead of $y^b$.

\begin{remark}\label{rmk_epsilon_j}
If we fix $t_0\in\R/T_j\Z$ and $r_0>0$ small enough then $\epsilon_j$ is the algebraic intersection number of the loop $\theta \in [0,2\pi] \mapsto (t_0,r_0,\theta)$ with the class $b$.
\end{remark}

Denote by $\P_{X_L}(D_L)$ the set of Borel probability measures on $M_L$ supported in~$D_L$ which are invariant by the flow of $X_L$. We claim that 
\begin{equation}\label{basic_assumption}
\left< c_\mu,\beta \right> > 0 \ \ \ \forall \mu\in \P_{X_L}(D_L).
\end{equation}

\begin{lemma}\label{lemma_rotation_numbers}
If $\mu\in\P_{X_L}(D_L)$ is supported in $\Sigma_j$ then $\rho^{y}(\gamma_j) = \frac{T_j}{2\pi} \int_{M_L} \beta(X_L) \ d\mu$.
\end{lemma}

\begin{proof}
As in~\eqref{vector_field_on_boundary} the vector field $X_L|_{\Sigma_j}$ is written as $\partial_t + b(t,\theta)\partial_\theta$. We lift it as a periodic vector field on the universal covering $\R^2$, where $t$ and $\theta$ lift to real-valued coordinates. Let $\theta(t;t_0,\theta_0)$ be the unique solution to $\dot\theta=b(t,\theta)$ with value $\theta_0$ at time $t=t_0$. The flow of $X_L$ on $\Sigma_j$ is $\phi^t(t_0,\theta_0) = (t+t_0,\theta(t+t_0;t_0,\theta_0))$, modulo $T_j\Z\times 2\pi\Z$. It follows from the argument in~\cite[pp. 104-105]{arnold}, and from the definition of $\rho^y(\gamma_j)$, that the sequence of functions
\[
g_n(t_0,\theta_0) = \frac{T_j}{2\pi} \frac{\epsilon_j}{2\pi} \ \frac{\theta(t_0+nT_j;t_0,\theta_0)-\theta_0}{nT_j} 
\]
converges to $\rho^{y}(\gamma_j)$ uniformly in $(t_0,\theta_0) \in \R^2$ as $n\to\infty$. By periodicity of $b$, the $g_n$ descend to functions on $\Sigma_j \simeq \R/T\Z \times \{0\}\times \R/2\pi\Z \to \R$. We compute
\begin{equation*}
\begin{aligned}
\rho^y(\gamma_j) &= \int_{\Sigma_j} \rho^y(\gamma_j) \ d\mu = \lim_{n\to\infty} \int_{\Sigma_j} g_n \ d\mu \\
&= \frac{T_j}{2\pi} \lim_{n\to\infty} \int_{\Sigma_j} \frac{1}{nT_j} \int_{0}^{nT_j} \beta(X_L) \circ \phi^t \ dt \ d\mu \\
&= \frac{T_j}{2\pi} \lim_{n\to\infty} \frac{1}{nT_j} \int_{0}^{nT_j} \int_{\Sigma_j} \beta(X_L) \circ \phi^t \ d\mu \ dt \\
&= \frac{T_j}{2\pi} \lim_{n\to\infty} \frac{1}{nT_j} \int_{0}^{nT_j} \int_{\Sigma_j} \beta(X_L) \ d\mu \ dt = \frac{T_j}{2\pi} \int_{\Sigma_j} \beta(X_L) \ d\mu
\end{aligned}
\end{equation*}
\end{proof}

Let $\mu\in\P_{X_L}(D_L)$ be arbitrary. For every Borel set $E\subset M_L$ define
\begin{equation*}
\begin{aligned}
& \mu_j(E) = \mu(E\cap \Sigma_j) \\
& \dot\mu(E) = \mu(E\cap {\rm int}(D_L)) = \mu(E\cap (M\setminus L))
\end{aligned}
\end{equation*}
Then $\dot\mu$ and the $\mu_j$ are $X_L$-invariant Borel measures, and $\mu=\dot\mu+\sum_j\mu_j$. We have
\begin{equation}\label{sum_decomp_measures}
\left< c_\mu,\beta \right> = \int_{M_L} \beta(X_L) \ d\mu = \int_{M_L} \beta(X_L) \ d\dot\mu + \sum_{j=1}^h \int_{M_L} \beta(X_L) \ d\mu_j
\end{equation}
If $\mu_j(M_L) = \mu(\Sigma_j)>0$ then $\mu_j/\mu_j(M_L) \in \P_{X_L}(D_L)$ is supported in $\Sigma_j$. Lemma~\ref{lemma_rotation_numbers} and hypothesis (iii) in Theorem~\ref{thm_SFS} together give
\[
\int_{M_L} \beta(X_L) \ d\mu_j = \int_{\Sigma_j} \beta(X_L) \ d\mu_j = \frac{2\pi}{T_j} \mu_j(M_L) \rho^y(\gamma_j) > 0
\]
in this case. If $\dot\mu(M_L) = \mu(M\setminus L) >0$ then $\dot\mu/\dot\mu(M_L)$ induces an element of $\Pp_\phi(M\setminus L)$, and by (iii) and Corollary~\ref{coro_int} we have 
\[
\int_{M_L} \beta(X_L) \ d\dot\mu = \int_{M\setminus L} \beta(X) \ d\dot\mu > 0
\]
in this case. Thus each term in the sum~\eqref{sum_decomp_measures} is non-negative, and at least one term is positive since $\mu(M_L)=\mu(D_L)=1$. We proved~\eqref{basic_assumption}.

As explained in~\ref{sssec_compactness} there exists $\omega \in \Omega^1(M_L)$ such that $\left<\cdot,\omega\right> > 0$ on $\Cc\setminus\{0\}$, and $K = \{c\in\Cc \mid \left<c,\omega\right>=1 \}$ is compact and convex. Let $c\in\Cc\setminus\{0\}$ be a cycle. In~\ref{sssec_representing_by_measures} it is proved that $c=\int_{M_L}X_L \ d\nu$ for some positive finite Borel measure~$\nu$ supported on~$D_L$, and it is easy to see that $\nu$ must be $X_L$-invariant because $c$ is a cycle. In other words, $\mu := \nu/\nu(M_L) \in \P_{X_L}(D_L)$ and $c = \nu(M_L) c_\mu$ for a Schwartzman cycle $c_\mu$. From~\eqref{basic_assumption} we conclude that $\left<c,\beta\right> = \nu(M_L) \left< c_\mu,\beta \right> > 0$. In particular $\Cc \cap B_1 = \{0\}$, or equivalently $K \cap B_1 = \emptyset$, and $\beta$ evaluates positively on $K\cap Z_1$. By Theorem~\ref{thm_refinement_HB} we find $\eta \in C_1'$ that vanishes on $B_1$, is positive on $K$, and agrees with $\beta$ on $Z_1$. By reflexivity $C_1'=\Omega^1(M_L)$ we conclude that $\eta$ is a $1$-form, and as such it must be closed since it vanishes on $B_1$. Moreover, $\eta|_{M\setminus L}$ represents~$y$ since it agrees with $\beta$ on $Z_1$. Finally, note that
\begin{equation}\label{crucial_pointwise_ineq_SFS}
\eta(X_L)|_p = \left<\delta_p,\eta\right> > 0 \qquad \forall \ p \in D_L
\end{equation}
because $\left<\cdot,\eta\right>>0$ on $\Cc\setminus\{0\}$.

Consider inclusions $\iota:D_L\hookrightarrow M_L$ and $\iota_j:\Sigma_j \hookrightarrow M_L$. The periods of $\iota^*\eta$ are integers since $y$ comes from an integral class on $M\setminus L$ and the inclusion map $M\setminus L \hookrightarrow D_L$ induces isomorphism in cohomology. Using coordinates $(t,\theta)$ in $\R/T_j\Z \times 0\times \R/2\pi\Z=\Sigma_j$, then $\{dt,d\theta\}$ is a basis of $H^1(\Sigma_j;\R)$, and $\iota_{j}^*\eta$ is homologous to $\epsilon_j d\theta/2\pi$.

The set of periods of $\iota^*\eta$ is equal to the set of all integers $\Z$. This is so because the class $y$ pulls back to $\pm d\theta/2\pi$ near each $\Sigma_j$. Choose $p_0\in M\setminus L$ and define a map ${\rm pr}:D_L \to \R/\Z$ by setting ${\rm pr}(p)$ to be the integral of $\iota^*\eta$ along any path from $p_0$ to $p$ modulo $\Z$
\[
{\rm pr}(p) = \int_{p_0}^p \iota^*\eta \mod \Z
\]
The map ${\rm pr}:D_L \to \R/\Z$ is a smooth surjective submersion in view of~\eqref{crucial_pointwise_ineq_SFS}. An important property that follows from this construction is that if $c:S^1 \to M\setminus L$ is a smooth loop then
\begin{equation}\label{periods_formula}
\left<y,c_*[S^1]\right> = \text{degree of ${\rm pr}\circ c$.}
\end{equation}

The preimages ${\rm pr}^{-1}(x)$ are the leaves of a foliation of $D_L$ obtained by integrating $\ker \iota^*\eta$. Each ${\rm pr}^{-1}(x)$ is a compact embedded submanifold of $D_L$ that intersects the boundary $\partial D_L$ cleanly, since it is transverse to $X_L$ and $X_L$ is tangent to $\partial D_L$. It follows that ${\rm pr}^{-1}(x) \subset D_L$ is a smooth embedded surface transverse to $X_L$ with boundary equal to $\pr^{-1}(x)\cap \partial D_L$.

Each leaf ${\rm pr}^{-1}(x)$ can be co-oriented by the vector field $X_L$, and can also be co-oriented by pulling back the canonical orientation of $\R/\Z$ via the map ${\rm pr}$. These co-orientations coincide, basically by construction. Equation~\eqref{periods_formula} implies that each ${\rm pr}^{-1}(x)$ induces the class $b$ seen in $H_2(M,L;\R)$. Later we will be in position to show that the same is true with $\Z$ coefficients.

We need to show that each ${\rm pr}^{-1}(x) \cap M\setminus L$ is the interior of a Seifert surface for $L$ in $M$ which is a global surface of section. Firstly, note that all trajectories in $D_L$ will hit all leaves ${\rm pr}^{-1}(x)$ in finite time both in the future and in the past: this follows from compactness of $D_L$ and from~\eqref{crucial_pointwise_ineq_SFS}. In particular, the return time is finite for all trajectories in $M\setminus L$.

The next step is to study the structure of the boundary of the leaves. We claim that for all $j$ and $x$ the submanifold ${\rm pr}^{-1}(x) \cap \Sigma_j$ is a circle dual to $\epsilon_jdt/T_j$. Let $\alpha$ be a connected component of ${\rm pr}^{-1}(x) \cap \Sigma_j$ oriented as a piece of the boundary of ${\rm pr}^{-1}(x)$. Then $\alpha$ is an embedded circle in $\Sigma_j$. It is non-trivial in $H_1(\Sigma_j;\Z)$ since otherwise it would bound a disk $D \subset \Sigma_j$ with $X_L\pitchfork \partial D$, thus forcing a singularity of $X_L$ on $\Sigma_j$. Let $\{e_1,e_2\}$ be a basis in $H_1(\Sigma_j;\Z)$ dual to $\{dt/T_j,d\theta/2\pi\}$ and write $\alpha = n_1e_1+n_2e_2$ in homology. Since $\alpha$ is tangent to the kernel of $\iota_j^*\eta$ we get
\[
0 = \int_\alpha \iota_j^*\eta = \epsilon_jn_2 \qquad \Rightarrow \qquad n_1\neq0 \ \text{and} \ n_2=0.
\]
If $|n_1|>1$ then $\alpha$ would have self-intersections. Hence $n_1=\pm1$. One can use Remark~\ref{rmk_epsilon_j} to conclude that $n_1=\epsilon_j$.

So far we have proved that if ${\rm pr}^{-1}(x) \cap \Sigma_j$ is non-empty then its components are embedded circles dual to $\epsilon_j dt/T_j$. Hence to conclude our study of the boundary we need to show that ${\rm pr}^{-1}(x) \cap \Sigma_j$ is connected and non-empty. It has to be non-empty since $\Sigma_j$ is an invariant torus, hence if ${\rm pr}^{-1}(x)$ did not touch $\Sigma_j$ then trajectories in $\Sigma_j$ would not hit ${\rm pr}^{-1}(x)$, impossible. Let $s$ be the number of connected components of ${\rm pr}^{-1}(x) \cap \Sigma_j$, and consider the loop $c:\R/2\pi\Z \to \Sigma_j$ defined by $c(\theta) = (t_0,0,\theta)$ in polar coordinates near $\Sigma_j$. Then the degree $d$ of ${\rm pr} \circ c$ satisfies $|d|=s$. Using~\eqref{periods_formula} together with the fact that $y \simeq \epsilon_jd\theta/2\pi$ near $\Sigma_j$ we compute $$ d = \left<y,c_*[\R/2\pi\Z]\right> = \left<  \epsilon_jd\theta/2\pi,c_*[\R/2\pi\Z] \right> = \epsilon_j \qquad \Rightarrow \qquad s=1. $$

It only remains to be shown that leaves are connected. Fix $x\in\R/\Z$, denote the first return map ${\rm pr}^{-1}(x) \to {\rm pr}^{-1}(x)$ by $\psi$, and denote by
\begin{equation}
\Psi : D_L \setminus {\rm pr}^{-1}(x) \to {\rm pr}^{-1}(x)
\end{equation}
the map defined by requiring $\Psi(p)$ to be the first point in the future trajectory of $p \in D_L \setminus {\rm pr}^{-1}(x)$ that belongs to ${\rm pr}^{-1}(x)$. In other words, $\Psi(p)$ is the first hitting point in ${\rm pr}^{-1}(x)$. By transversality of $X_L$ to ${\rm pr}^{-1}(x)$, $\psi$ is a smooth diffeomorphism and $\Psi$ is a smooth submersion. For each connected component $Y$ of ${\rm pr}^{-1}(x)$ we set
\begin{equation*}
C_Y = \psi^{-1}(Y) \cup \Psi^{-1}(Y).
\end{equation*}
In other words, $C_Y$ is the set of points $p\in D_L$ which will first hit ${\rm pr}^{-1}(x)$ at a point of $Y$ in its future trajectory $\phi^{(0,+\infty)}(p)$.

We claim that ${\rm pr}^{-1}(x)$ has no boundaryless connected components. Our argument is indirect. Suppose that some connected component has no boundary. By compactness of ${\rm pr}^{-1}(x)$ we find finitely many boundaryless connected components $Y_1,\dots,Y_N$ of ${\rm pr}^{-1}(x)$ such that $Y_j = \psi(Y_{j-1})$ for all $2\leq j\leq N$ and $Y_1 = \psi(Y_N)$. It follows that $C_{Y_1} \cup \dots \cup C_{Y_N}$ is open and closed in $D_L$, hence it is equal to $D_L$ since $D_L$ is connected. But this would force ${\rm pr}^{-1}(x)$ to be equal to $Y_1\cup \dots\cup Y_N$ and have no boundary, contradicting our standing assumption $L\neq\emptyset$.

The previously explained analysis of the boundary of ${\rm pr}^{-1}(x)$ told us that each~$\Sigma_j$ contains exactly one connected component of the boundary of ${\rm pr}^{-1}(x)$, which is a circle dual to $\epsilon_j dt/T_j$. It follows that if $p$ is a boundary point of ${\rm pr}^{-1}(x)$ then $\psi(p)$ is in the same connected component of the boundary as $p$. Since we already proved that all connected components of ${\rm pr}^{-1}(x)$ have non-empty boundary, we can conclude that $\psi$ preserves each connected component of ${\rm pr}^{-1}(x)$. It follows that $C_Y$ is open and closed in $D_L$ for each connected component $Y\subset {\rm pr}^{-1}(x)$. By connectedness of $D_L$, each such $C_Y$ is either empty or equal to $D_L$. Hence ${\rm pr}^{-1}(x)$ is connected.

Finally, we need show that the projection of ${\rm pr}^{-1}(x)$ to $M$ induces the class $b\in H_2(M,L;\Z)$. We already know this to be true with $\R$-coefficients. We argue as follows. Consider an embedded compact oriented surface $\hat B \subset D_L$ that intersects $\partial D_L$ cleanly, satisfies $\partial\hat B = \hat B \cap \partial D_L$, and projects to $M$ as a Seifert surface that represents $b$. Note that $\eta$ is exact on $\hat B$ since $y$ algebraically counts intersections of loops with $b$. Let $g:\hat B \to \R$ be a primitive of $\eta$ on $\hat B$. We may assume that $g(z_0)=0$ for some $z_0 \in \hat B\setminus\partial\hat B$. Denote $x_0 = {\rm pr}(z_0)$. Then ${\rm pr}(z) = x_0 + g(z) \mod \Z$ for all $z\in\hat B$. Consider the maps induced by inclusions
\[
\begin{aligned}
&\hat\iota: H_2(\hat B,\partial\hat B;\Z) \to H_2(D_L,\partial D_L;\Z) \\ 
&\iota_0: H_2({\rm pr}^{-1}(x_0),\partial {\rm pr}^{-1}(x_0);\Z) \to H_2(D_L,\partial D_L;\Z)
\end{aligned}
\]
Consider also $\hat X_L$ the vector field $X_L$ normalized so that $d{\rm pr}(\hat X_L)=1$. Then we can use the flow of $\hat X_L$ to deform $\hat B$ to ${\rm pr}^{-1}(x_0)$ by flowing each point from time zero to time $-g$. 
We conclude that the image $\hat\iota[\hat B]$ of the fundamental class $[\hat B]$ belongs to the image of $\iota_0$, i.e. $\hat\iota[\hat B] = m \ \iota_0[{\rm pr}^{-1}(x_0)]$ for some $m\in \Z$. Using now that $y$ also algebraically counts intersections of loops with ${\rm pr}^{-1}(x_0)$ we get $m=1$. The desired conclusion follows.

We are done with the proof of (iii) $\Rightarrow$ (ii) in the case $L\neq\emptyset$. \\

\subsection{Proof that (i) $\Rightarrow$ (iii) holds $C^\infty$-generically}\label{ssec_necessary}

Denote by $\rho_j$ the rotation number of $\gamma_j$ with respect to $y^b$. The $C^\infty$-generic assumption needed for the proof reads as follows: {\it For every $j$, if $\rho_j=0$ then $\gamma_j$ is hyperbolic.}

Let $S$ be a global surface of section, oriented by the ambient orientation and the vector field $X$ as usual, representing the class $b\in H_2(M,L)$ and satisfying $\partial S = L$. Denote by $\iota:S\to M$ the inclusion map. In particular, every trajectory contained in $M\setminus L$ will hit ${\rm int}(S) = S\setminus \partial S = S\setminus L$ infinitely often in the future and in the past.

Let $(\tau,s) \in \R/ T_j\Z \times [0,1)$ be tubular coordinates on $S$ near $\gamma_j \subset L = \partial S$ such that $\gamma_j \simeq \R/T_j\Z \times 0$ and the map $\tau \mapsto \iota(\epsilon_j\tau,0)$ is a flow-parametrization of $\gamma_j$. Here $\epsilon_j=+1$ if the orientations of $\gamma_j$ induced by $S$ and by the flow coincide, or $\epsilon_j=-1$ otherwise. In particular, $\tau \mapsto \iota(\tau,0)$ parametrizes $\gamma_j$ as the boundary of~$S$, and the orientation of $S$ is given by $d\tau\wedge ds$. 

Choose tubular polar coordinates $(t,r,\theta)$ around $\gamma_j$ such that $\iota(\tau,s) = (\epsilon_j\tau,s,0)$. As in~\ref{ssec_blowing_up} $X = X(t,r,\theta)$ can be written in these coordinates as
\begin{equation}\label{local_rep_X}
X = \partial_t+ b(t,\theta)\partial_\theta + \epsilon(t,r,\theta) 
\end{equation}
with $|\epsilon|=O(r)$ as $r\to0$. Denoting $X = X^t\partial_t + X^r\partial_r + X^\theta\partial_\theta$ near $\gamma$, we have  
\begin{equation}\label{positive_theta_component}
\text{$\epsilon_j X^\theta(t,r,0) = \epsilon_jb(t,0) + O(r) > 0$ for $r>0$ small}
\end{equation}
by transversality of the flow with ${\rm int}(S)$. Taking the limit as $r\to0^+$ we get
\begin{equation}\label{infinitesimal_ineq}
\epsilon_jb(t,0)\geq 0, \ \ \ \forall t
\end{equation}
where $b(t,\theta)$ is the function~\eqref{vector_field_on_boundary}.

Using~\eqref{positive_theta_component} we find a smooth vector field $Y$ on $M\setminus L$ which is transverse to ${\rm int}(S)$, induces the same co-orientation of $S$ as $X$, and coincides with $\epsilon_j\partial_\theta$ near $\gamma_j$. Using the flow of $Y$ we construct a diffeomorphism between a neighborhood $V$ of ${\rm int}(S)$ in $M\setminus L$ and $(-\delta,\delta) \times {\rm int}(S)$, with $\delta>0$ small, such that 
\begin{itemize}
\item ${\rm int}(S) \simeq 0 \times {\rm int}(S)$.
\item If $h$ denotes the coordinate on $(-\delta,\delta)$ then $Y = \partial_h$ on $V$, and $h=\epsilon_j\theta$ near~$\gamma_j$ with respect to the chosen polar tubular coordinates.
\end{itemize}
Choose $\chi:(-\delta,\delta)\to[0,+\infty)$ a smooth compactly supported bump function satisfying $\chi(0)>0$ and $\int_{-\delta}^\delta \chi(h) dh = 1$ and define $\eta\in\Omega^1(M\setminus L)$ by $\eta = \chi(h)dh$ on $V$ and $\eta=0$ outside of~$V$. Clearly $\eta \in \Omega^1_L$, $\eta$ is closed since $d(\chi dh) = \chi' dh \wedge dh = 0$, $\eta$ represents $y^b$ and $\eta(X)$ is positive on ${\rm int}(S)$.

By definition $\epsilon_j\rho_j$ is equal to a positive multiple of the rotation number of the dynamics of the non-autonomous equation $\dot \theta(t) = b(t,\theta(t))$ on the circle $\R/2\pi\Z$. Hence~\eqref{infinitesimal_ineq} implies that $\rho_j\geq0$ for all $j$. Assume, by contradiction, that $\rho_j=0$. Then we find a fixed point $\theta_0$ of this dynamical system. This means that if $\theta(t;\theta_1)$ denotes the solution satisfying $\theta(0;\theta_1)=\theta_1$ then $\theta_0$ is a fixed point of the map $\theta_1 \mapsto \theta(T_j;\theta_1)$. Moreover, since $\rho_j=0$, the total variation of a lift of $\theta(t;\theta_0)$ vanishes when $t$ varies from $0$ to $T_j$. Since $\gamma_j$ is hyperbolic there will be an $X$-invariant strip contained either on the stable or on the unstable manifold of $\gamma_j$ which is asymptotic to the loop $\Gamma = (t,0,\theta(t;\theta_0))$ on the torus $\R/T_j\Z \times \{0\} \times \R/2\pi\Z$. A trajectory on this strip will hit ${\rm int}(S)$ infinitely often by the assumption that $S$ is a global surface of section. In particular, by transversality of $X$ with $S\setminus\partial S$, a lift of the $\theta$ coordinate to $\R$ would oscillate very much along this trajectory. However, since $\Gamma$ has slope $(1,0)$, this oscillation is bounded, contradiction. We conclude that $\rho_j>0$ for every $j$.

Consider the first return time function $\tau : {\rm int}(S) \to (0,+\infty)$ defined by the identity $\tau(p) = \inf \{ t>0 \mid \phi^t(p) \in S \}$. We claim that
\begin{equation}\label{tau_bounded_from_above}
\sup \{ \tau(p) \mid p \in {\rm int}(S) \} < +\infty.
\end{equation}
Fix $j$ and work near $\gamma_j$. For every $t_0 \in [0,T_j]$ let $f(t;t_0)$ denote the solution of the equation $\dot f(t) = b(t,f(t))$, $f(t_0)=0$. Then $$ \rho_j = \lim_{t\to+\infty} \frac{|f(t;t_0)|}{t} = \lim_{t\to+\infty} \frac{|f(t+t_0;t_0)|}{t} $$ with the convergence being uniform in $t_0 \in [0,T_j]$. Choose $s_0>0$ such that $|f(t_0+s_0;t_0)| > s_0\rho_j/2 \geq 3\pi$ for every $t_0 \in [0,T_j]$. Choose $r_0>0$ such that if $0<r<r_0$ then a trajectory of $\phi^t$ starting in $\{r<r_0\} \cap {\rm int}(S)$ is contained in the domain of the polar tubular coordinates up to time $s_0$. Note that in coordinates $(t,r,\theta)$ the vector field $X$ extends smoothly to the torus $\{r=0\}$ by the formula $\partial_t+b(t,\theta)\partial_\theta$. Hence we find $0<r_1<r_0$ such that if $0 < r < r_1$ then a lift $g(t)\in \R$, $t\in[0,s_0]$, of the $\theta$-component of the trajectory starting at $(t_0,r,0)$ will be close to $f(t+t_0;t_0)$ uniformly in $t\in[0,s_0]$, in fact we have an estimate $$ \sup \left\{ |g(t)-f(t+t_0;t_0)| \mid t_0 \in \R/T_j\Z, \ t\in[0,s_0] \right\} = O(r) \ \text{as} \ r\to0^+. $$ It follows that $|g(s_0)|>2\pi$ for all points in ${\rm int}(S)\cap \{r<r_1\}$. We get a bound $\tau \leq s_0$ for these points. Repeating this argument for all $j$, we conclude that $\tau$ is uniformly bounded from above on the ends of ${\rm int}(S)$. Hence~\eqref{tau_bounded_from_above} holds.

Fix $\mu \in \Pp_\phi(M\setminus L)$. By Corollary~\ref{coro_int} we know that
\begin{equation*}
\mu \cdot y^b = \int_{M\setminus L} \eta(X) \ d\mu
\end{equation*}
and by the Ergodic Theorem (together with Lemma~\ref{lemma_bounded_intersection_SFS}) we know that 
\begin{equation}
\lim_{T\to+\infty} \frac{1}{T} \int_{\phi^{[0,T]}(p)} \eta
\end{equation}
exists for $\mu$-almost every point $p$, and defines a $\mu$-integrable function with integral equal to $\mu \cdot y^b$. Set $a = \sup\{\tau(p)\mid p\in{\rm int}(S) \}$. By the construction of $\eta$ we see that 
\begin{equation*}
\begin{aligned}
\liminf_{T\to+\infty} \frac{1}{T} \int_{\phi^{[0,T]}(p)} \eta &\geq \liminf_{T\to+\infty} \frac{1}{T} \#\{0 \leq t \leq T \mid \phi^t(p) \in S\} \geq \liminf_{T\to+\infty} \frac{1}{T} \left\lfloor \frac{T}{a} \right\rfloor = \frac{1}{a}
\end{aligned}
\end{equation*}
holds for every $p \in M\setminus L$. We conclude that $\mu\cdot y^b \geq 1/a > 0$.

\appendix

\section{A small refinement of Hahn-Banach}

\begin{theorem}\label{thm_refinement_HB}
Suppose
\begin{itemize}
\item $X$ is a locally convex topological vector space over $\R$
\item $Z$ is a closed linear subspace
\item $K\subset X$ is compact and convex
\item $f:Z \to \R$ is linear and continuous, and satisfies $f|_{K \cap Z}>0$
\end{itemize}
Then there exists $F:X \to \R$ such that
\begin{itemize}
\item[(a)] $F$ is linear and continuous
\item[(b)] $F|_Z=f$ and $F|_K>0$
\end{itemize}
\end{theorem}

\begin{proof}
If $K$ is empty then this is a direct consequence of the standard Hahn-Banach theorem (analytic version). 

Assume from now on that $K\neq\emptyset$. If $f\equiv0$ then $K\cap Z = \emptyset$ and we can again apply the standard Hahn-Banach theorem (geometric version) to find a linear continuous $g:X\to\R$ and $\epsilon>0$ such that $g|_K\geq\epsilon$ and $g|_Z<\epsilon$. This forces $Z \subset \ker g$ and we are done with this case as well. We proceed assuming $f\neq0$ and $K\neq\emptyset$.

Choose $u\in Z$ such that $f(u)=1$ and consider the convex hull $K'$ of $K \cup \{u\}$. Then $K'$ is compact and convex. Note that $f$ is positive on $K'\cap Z$, in fact if $x\in K'\cap Z$ then $x = (1-t)u+ty$ for some $t\in[0,1]$ and $y\in K\cap Z$, and $f(x) = (1-t) + tf(y) > 0$ since $f(y)>0$. The standard Hahn-Banach theorem (geometric version) provides $g:X\to\R$ linear continuous and $\epsilon>0$ such that $g|_{\ker f}<\epsilon$ and $g|_{K'}\geq\epsilon$. It follows that $\ker f \subset \ker g$ and $g(u)>0$, in particular there exists $c>0$ such that $g|_Z = cf$ (take $c=g(u)$). The proof is finished if we set $F = \frac{1}{c}g$.
\end{proof}

\end{document}